\documentclass[12pt]{amsart}
\usepackage{amsmath,amsthm,amsfonts}
\newtheorem{thm}{Theorem}[section]
\newtheorem{cor}[thm]{Corollary}
\newtheorem{lem}[thm]{Lemma}
\newtheorem{prop}[thm]{Proposition}

\theoremstyle{definition}

\def\e{{\rm e}}

\begin{document}

\title[Spectra of averages of unitrary representations]
{Spectra of averages\\of unitrary representations of LCA groups}

\author{Guy Cohen}
\address{School of Electrical Engineering, Ben-Gurion University, Beer-Sheva, Israel}
\email{guycohen@bgu.ac.il}

\author{Michael Lin}
\address{Department of Mathematics, Ben-Gurion University, Beer-Sheva, Israel}
\email{lin@math.bgu.ac.il}


\begin{abstract}
Let $G$ be a locally compact Abelian (LCA) group with dual group $\Gamma$,
 and let $\mu$ be a probability measure on (the Borel sets of) $G$. 
Given a unitary representation $\{U(t): t \in G\}$  in a complex Hilbert space $H$,
we study the spectrum of the $\mu$-average $V:=\int_G U(t)d\mu(t)$ (defined in the
strong topology of $H$). We prove that $\sigma(V) \subset \overline{\hat\mu(\Gamma)}$
and give a sufficieent condition for equality.

Using the spectral measure $E(\cdot)$ given by the general Stone theorem, we prove
a (weak)  spectral mapping theorem for the operators $U(\nu):=\int_GU(t)d\nu(t)$,
where $\nu$ is any bounded complex measure on $G$.

For a unitary representation of $\mathbb Z$, defined by the powers of a unitary
operator $U$, we prove that $\sigma(V)=\hat\mu(\sigma(U))$.
For a unitary representation of $\mathbb R$, given as $U(t)=\e^{itB}$ ($t\in\mathbb R$),
we show that $\sigma(V)=\overline{\hat\mu(\sigma(B))}$.
\end{abstract}

\dedicatory{\large Dedicated to the memory of Joe Rosenblatt}
\maketitle

\section{Introduction}

\subsection{Background}

A {\it general dynamical system} is an action of a locally compact (often metrizable) group 
$G$ (with unit $e$) by invertible measure preserving transformations $\{\theta(t): t \in G\}$ 
on a probability space $(\Omega,\mathcal B,m)$, such that 
$\theta(t_1t_2)=\theta(t_1)\theta(t_2)$, $\theta(e)$ is the identity, and $\theta(t)x$ is 
measurable from $G\times\Omega$ to $\Omega$.
 Such an action gives rise to a unitary representation of $G$ in $L^2(m)$, defined by 
$U(t)f =f\circ \theta(t^{-1})$. When $L^2(m)$ is separable, this measurable representation
is strongly continuous.

Given a "general dynamical system", Oseledets \cite{Os} studied the ergodic properties of 
the Markov operator $P_\mu:=\int_G U(t)d\mu(t)$, where $\mu$ is a probability measure 
on $G$; he proved  a.e. convergence of $P_\mu^n f$, $f \in L^1(m)$, when $\mu$ is
 symmetric strictly aperiodic.
Bellow, Jones and Rosenblatt \cite{BJR},\cite{BJR1} looked at actions of $\mathbb Z$
and studied a.e convergence of $P_\mu^n f$ when $\mu$ is not symmetric.

Jones, Rosenblatt and Tempelman \cite{JRT} extended some of the results of \cite{BJR1}
to actions of LCA groups. For that they proved that
 $\sigma(P_\mu) \subset \overline{\hat\mu(\Gamma)}$, where $\Gamma$ is the dual group
of $G$.
\medskip

\subsection{Notations and overview}

Throughout this work, $G$ is a locally compact Abelian (LCA) group, with fixed Haar 
measure $m_G$ and dual group $\Gamma$. We denote by $M(G)$ the space of bounded complex regular measures on  $G$, and identify the space of absolutely contiuous 
elements of $M(G)$ with $L_1(G):=L_1(G,m_G)$.  The Fourier-Stieltjes transform of 
$\nu \in M(G)$ is defined by $\hat\nu(\gamma);=\int_G \gamma(t)d\nu(t)$,
 $\gamma \in \Gamma$.

The  spectrum of an operator $T$ on a complex Banach space is denoted $\sigma(T)$, 
with $\rho(T):=\sup_{\lambda \in \sigma(T)} |\lambda|$ its spectral radius.

Given a strongly continuous unitary representation $\mathbf{U}=U(t)_{t \in G}$, 
for $\nu \in M(G)$ we define the operator $U(\nu):= \int_G U(t)d\nu(t)$ (defined at 
least weakly).  In Section \ref{powers} we study the convergence of the powers of $U(\mu)$
when $\mu$ is a probability on $G$.

The main purpose of this work, in Section \ref{spectra}, is to obtain spectral properties of
the operators $U(\nu)$. The main tool for our investigations is $E(\cdot)$, the spectral 
measure of $\mathbf{U}$ given by the general Stone theorem (e.g. \cite{Am}).

\smallskip

For weakly continuous bounded representations $\mathbf{T}$ in Banach spaces,
the modern tool for studying spectral properties of $T(\nu):=\int_G T(t)d\nu(t)$ is
the Arveson spectrum \cite{Ar1} (see Subsection \ref{Arveson}).
The point of our work is that for unitary representations in complex Hilbert spaces, 
the study can be carried out {\it using only properties of spectral measures on
 $\Gamma$  (as in \cite{Ru}) and classical Fourier analysis on $G$ (as in  \cite{Ru1})}.

Let $E(\cdot)$ be a regular spectral measure defined on $\Gamma$, with support
$\Sigma(E)$. For a bounded measurable $f$ on $\Gamma$ the operator 
$\Psi(f):=\int_\Gamma f(\gamma)E(d\gamma)$ is defined weakly. When $f$ is bounded continuous 
on $\Gamma$, we prove the (weak) spectral mapping theorem (WSMT): 
$\sigma(\Psi(f))=\overline{f(\Sigma(E))}$; we deduce from it the WSMT for  $U(\nu)$,
 for {\it every} $\nu \in M(G)$, which seems to be a new result.

We prove that $\Sigma(E)$ equals the Arveson spectrum (stated without proof in 
 \cite{St}); consequently we conclude that our WSMT for unitary representations is better 
than the general WSMT  \cite{DLZ},\cite{Ar2} for representations in Banach spaces,
which for certain representations may fail when $\nu$ is continuous singular \cite{DLZ}.
\medskip

\bigskip

\section{Convergence of the powers of representation averages} \label{powers}

When $T(t)_{t \in G}$ is a bounded continuous representation  in a reflexive Banach space
$X$ and $\mu$ is a probability on $G$, the operator $V=\int_G T(t)d\mu(t)$ is 
power-bounded, and by the mean ergodic theorem $\frac1N \sum_{k=1}^N V^k$ converges in the 
strong operator topology. In this section we look at conditions for strong convergence of 
$V^n$ when $X$ is a (complex) Hilbert space.

\subsection{Convergence in operator norm}

\begin{prop} \label{uet}
Let $T$ be a normal contraction  on a complex Hilbert space.  The averages
$\frac1n\sum_{k=1}^n T^k$ converge in operator norm if (and only if) $1$ is isolated in 
$\sigma(T)$.
\end{prop}
\begin{proof} When $1$ is isolated in the spectrum, 
 $\sup_{1 \ne \lambda \in \sigma(T)} |\sum_{k=1}^n \lambda^k| \le 2/\delta$ for some
 $\delta>0$, and the result follows from the spectral theorem.

 When $T$ on a complex Banach space $X$  has $\frac1n\sum_{k=1}^n T^k$ convergent in operator 
 norm, say to $E$, the ergodic decomposition $X=\{x\in X: Tx=x\}\oplus \overline{(I-T)X}$
 yields that $\|\frac1n\sum_{k=1}^n T^k\| \to 0$ on $Y:=\overline{(I-T)X}$. Hence
 $I-\frac1n\sum_{k=1}^n T^k$ is invertible on $Y$  for large $n$, so $(I-T)Y=Y$. Hence $1$ 
is isolated in $\sigma(T)$.
\end{proof}

\begin{prop} \label{UET}
Let $G$ be a LCA group with dual $\Gamma$ and $\mu$ a probability on $G$ such that $1$
is isolated in $\overline{\hat\mu(\Gamma)}$. Let $\{U(t):t \in G\}$ be a continuous
unitary representation in $H$, and $V:=\int_G U(t)d\mu(t)$.
Then $\frac1n\sum_{k=1}^n V^k$ converges in operator norm.

If in addition $\overline{\hat\mu(\Gamma)}\cap \mathbb T =\{1\}$, then $V^n$ converges 
in operator norm.
\end{prop}
\begin{proof} By the proof of \cite[Proposition 1.9]{JRT} (see Theorem \ref{jrt} below),
$\sigma(V)\subset \overline{\hat\mu(\Gamma)}$, so $\frac1n\sum_{k=1}^n V^k$ converges in 
operator norm by Proposition \ref{uet}.
 
If in addition $\overline{\hat\mu(\Gamma)}\cap \mathbb T =\{1\}$, then 
$\sigma(V) \subset \{1\}\cap\{z: |z|\le \rho\}$ for some $\rho<1$, and 
$\|V^n\|_{I-V)H} \to 0$, so $V^n$ converges in operator norm by the ergodic decomposition.
\end{proof}

{\bf Remarks.} 1.  Proposition \ref{UET} applies when $G$ is compact and $\mu \in M_0(G)$,
since $\Gamma$ is discrete; see also \cite[Corollary 4.16]{GJ}.

2. In general, if $\mu$ is absolutely continuous and supported in a compact subgroup,
then Proposition \ref{UET} applies, by \cite[Theorem 6.6]{GJ}.

3. Galindo and Jord\'a \cite[Corollary 4.11(ii)]{GJ} proved Proposition \ref{UET} for
the regular representation in $L_2(G)$, and studied convergence in operator norm of 
$\frac1n \sum_{k=1}^n V^k$ for the regular representation in $L_p(G)$, $1<p<\infty$.
\medskip

\subsection{Ritt operators}

An  operator $T$ on a complex  Banach space $X$ is called 
{\it Ritt} if it is power-bounded and $\sup_n n\|T^n(I-T)\| < \infty$.

When $T$ is Ritt on $X$ reflexive, $T^n$ converges strongly, by the ergodic decomposition.
\smallskip

{\bf Definition.} Let $0\le r<1$. The closed convex hull of the disk of radius $r$ centered 
at 0 with the point 1 is called a {\it closed Stolz region}. It is obtained by adding to the 
disk the closed area between the two tangent lines to the disk originating at 1. 
See sketch in \cite[p. 110]{BJR} or \cite[p. 194]{LLM}.
If $\Delta$ is a Stolz region, then there is a $C$ such that $|1-\zeta| \le C(1-|\zeta|)$  
for every $\zeta \in \Delta$ \cite[p. 4]{Pr}. See \cite{GT} for additional properties. 
\smallskip

Nagy and Zem\'anek \cite{NZ} proved that the spectrum of a Ritt operator is contained
in a closed Stolz region; see  also Lyubich \cite{Ly}. 
The converse is false, even for contractions on Hilbert space \cite[p. 1133]{CCL}. 
However, if $T$ is power-bounded with $\sigma(T)$ contained in a Stolz region, then 
$\|T^n(I-T)\| \to 0$ by \cite{KT}.

\begin{prop} \label{stolz}
{\bf \cite[Proposition 2.5]{CCL}} A normal contraction on a complex Hilbert space is Ritt 
if and only if its spectrum is contained in a Stolz region.
\end{prop}

The inportance of the Rittt condition in ergodic theory is shown by the following result of
Le Merdy and Xu \cite[Corollary 5.2]{LMX} (see also \cite[Theorem 2.7]{CCL}).

\begin{thm} \label{lmx}
Fix $1<p< \infty$, and let $T$ be a positive contraction of $L^p(\Omega,\mathcal B,m)$.
If $T$ is Ritt, then $T^nf$ converges a.e. for any $f \in L^p(\Omega,m)$.
\end{thm}

 \begin{thm} \label{llm}
Let $G$ be a LCA group with dual $\Gamma$, and $\mu$ a probability on $G$ satisfying
 the following {\rm bounded angular ratio (BAR)} condition: there exists $C$ such that
\begin{equation} \label{bar}
\sup_{1 \ne \gamma\in \Gamma} \frac{|1-\hat\mu(\gamma)|}{ 1-|\hat\mu(\gamma)|}=C < \infty.
\end{equation}
If $1<p<\infty$ (fixed), and $\{S(t): t\in G\}$ is a strongly continuous representation
of $G$ by positive isometries of $L^p(\Omega,\mathcal B,\mu)$, then the operator 
$P:=\int_G S(t)d\mu(t)$ is Ritt, and $P^nf$ converges a.e for any $f \in L^p(\Omega,m)$.
\end{thm}
\begin{proof} The operator $P$ is Ritt by \cite[Theorem 6]{LLM}, and is a positive contraction
by the properties of the representation $S(t)$. Hence the a.e. convergence of $P^nf$ holds
by Theorem \ref{lmx}.
\end{proof}

\begin{cor}{\cite[Theorem 2.20.]{JRT}}
Let $G$ be a LCA group with dual $\Gamma$ and $\mu$ a probability on $G$ satisfying 
\eqref{bar}. Let $U(t)$ be the representation of $G$ in $L^p(\Omega,m)$  induced by
a measure preserving action in a probability space $(\Omega,\mathcal B,m)$, and define 
the operator $V:=\int_G U(t)d\mu(t)$. Then $V^nf$ converges a.e. for any 
$f \in L^p(\Omega,m)$, $1<p<\infty$.
\end{cor} 

\medskip

\subsection{Strict aperiodicity} 

 A probability $\mu$ on a LCA group $G$ is called {\it strictly aperiodic}
if its support $\mathcal S$ is not contained in a class of a closed subgroup different 
from $G$. Equivalently, $G$ is the closed group generated by $\mathcal S - \mathcal S$. 
If $\mu$ is strictly aperiodic, then it is {\it adapted} -- the closed subgroup generated by
$\mathcal S$ is $G$.

It is welll-known that $\mu$ is adapted if and only if $\hat\mu(\gamma) \ne 1$ for 
every $1 \ne \gamma \in \Gamma$.
It  is also known (\cite[Theorem 3.3]{BJR0}) that a probability $\mu$ is strictly aperiodic
 if and only if $|\hat\mu(\gamma)|<1$ for every $1 \ne \gamma \in \Gamma$. 
Hence an adapted probability satisfying \eqref{bar} is strictly aperiodic.

\smallskip

It was shown in \cite[Theorem 4]{Lin} that if  $\mu$ is strictly aperiodic on a LCA group
$G$, then for any bounded continuous representation $T(t)$ in a reflexive Banach space,
the powers of the operator $V:=\int_G T(t)d\mu(t)$ converge in the strong operator topology.
Conversely, if $V^n$ converges weakly for every unitary representation of $G$, then $\mu$
is strictly aperiodic, by \cite[Theorem 2.1]{LW}.

Given a power-bounded operator $V$ on a reflexive Banach space, the convergence 
$\|V^n(I-V)\| \to 0$ implies
convergence of $V^n$ in the strong operator topology, by the  ergodic decomposition. 

\begin{prop} \label{foguel-weiss}
Let $G$ be a LCA group with dual $\Gamma$ and $\mu$ a probability on $G$.
Let $U(t)$ be a continuous unitary representation and define $V=\int_G U(t)d\mu(t)$.
If $\mu(\{0\})>0$, then $\|V^n(I-V)\| =O(1/\sqrt{n})$.
\end{prop}
\begin{proof} Put $\alpha=\mu(\{0\})>0$. If $\alpha=1$, then $V=U(0)=I$ and 
$V^n(I-V)=0$.

Let $\alpha<1$. Put $\mu_1:=(1-\alpha)^{-1}\mu_{|G\setminus\{0\}}$ and
$V_1=\int_G U(t)d\mu_1(t)$. Then $V=\alpha I +(1-\alpha) V_1$. Since
$I-V=(1-\alpha)(I-V_1)$, by \cite[Lemma 2.1]{FW} 
$\|V^n(I-V)\|=(1-\alpha)\|V^n(I-V_1)\|=O(1/\sqrt{n})$.
\end{proof}

{\bf Remarks.} 1. If $\mu$ is adapted and $0$ is in its support, then $\mu$ is strictly 
aperiodic.
\smallskip

2. In general $V$ is not Ritt. For a simple example, let $U$ be a unitary operator
with $\sigma(U)=\mathbb T$, and $\mu=\frac12 \delta_0 +\frac12\delta_1$ on
$\mathbb Z$. Then for $V= \frac12 I+\frac12 U$, $\sigma(V)$ is the circle
 $\{\lambda: |\lambda-\frac12|=\frac12\}$, which is not contained in any Stolz region.
 Since 1 is not isolated in its spectrum, $V$ is not uniformly ergodic.

\begin{prop} \label{ac}
Let $G$ be a LCA group with dual $\Gamma$ and $\mu$ a strictly aperiodic
probability on $G$.
Let $U(t)$ be a continuous unitary representation and define $V=\int_G U(t)d\mu(t)$.
If $\mu$ is absolutely continuous ($\mu \ll m_G$), then $\|V^n(I-V)\| \to 0$.
\end{prop}
\begin{proof} It was proved by Derriennic \cite[Proposition 6]{De} that strict aperiodicity
of $\mu$ is equivalent to the convergence $\|f*\mu^n\|_1 \to 0$ for every $f \in L_1(G)$
with $\int f\,dm_G=0$. When $\mu \ll m_G$, we put $f :=d(\mu-\mu^2)/dm_G$
and obtain $ \|\mu^n-\mu^{n+1}\|_1 =  \|\mu^{n-1}*f\|_1 \to 0$. Then
$\|V^n(I-V)\|=\|\int U(t)d(\mu^n-\mu^{n+1})\| \le \|\mu^n-\mu^{n+1}\|_1 \to 0$.
\end{proof}
\medskip

The following is a very special case of the deep Katznelson-Tzafriri theorem \cite{KT}.
\begin{prop} \label{kt}
Let $T$ be a normal contraction on a complex Hilbert space. Then $\|T^n(I-T)\| \to 0$
if and only if $\sigma(T)\cap \mathbb T\subset\{1\}$.
\end{prop}
\begin{proof} Assume $\|T^n(I-T)\| \to 0$. For $\lambda \in \sigma(T)$ we have
 by the spectral mapping theorem
$|\lambda^n(1-\lambda)| \le \rho(T^n(I-T)) \le \|T^n(I-T)\| \to 0$, 
so if $|\lambda|=1$ then $\lambda=1$.

Assume  $\sigma(T)\cap \mathbb T\subset\{1\}$. Since $T$ is normal,
 by \cite[Corollary X.2.8(ii)]{DS}
 $$
\|T^n(I-T)\| \le \sup_{\lambda \in \sigma(T)} |\lambda^n(1-\lambda)|,
$$
which converges to zero by the assumption.
\end{proof}


\begin{prop} \label{kt-V}
Let $G$ be a LCA group with dual $\Gamma$ and $\mu$ a probability on $G$, such that
$\overline{ \{\hat\mu(\gamma): \gamma \in \Gamma\} } \cap \mathbb T=\{1\}$. 
Let $U(t)$ be a continuous unitary representation and define $V=\int_G U(t)d\mu(t)$.
Then $\|V^n(I-V)\| \to 0$.
\end{prop}
\begin{proof} Clearly $V$ is a normal contraction. By \cite{JRT} (see Theorem \ref{jrt} 
below), $\sigma(V) \subset \overline{\{\hat\mu(\gamma): \gamma\in \Gamma\}}$. 
By the assumption, $\sigma(V)\cap \mathbb T=\{1\}$, so $\|V^n(I-V)\| \to 0$ by 
Proposition \ref{kt}.
\end{proof}

{\bf Remark.} When $\mu$ is adapted, the assumption on $\mu$ in Proposition \ref{kt-V}
 implies strict  aperiodicity, but is strictly stronger.

\begin{cor} \label{M-0}
Let $G$ be a LCA group with dual $\Gamma$ and $\mu$ a strictly aperiodic
probability on $G$.
Let $U(t)$ be a continuous unitary representation and define $V=\int_G U(t)d\mu(t)$.
If $\mu \in M_0(G)$, i.e. $\hat\mu \in C_0(\Gamma)$, then $\|V^n(I-V)\| \to 0$.
\end{cor}
\begin{proof} let $\Delta$ be a compact subset of $\Gamma$ such that
$|\hat\mu(\gamma)| <\frac12$ for $\gamma \notin \Delta$. Since $\mu$ is strictly
aperiodic, the compact set $\hat\mu(\Delta)$ intersects $\mathbb T$ only at 1.
Since $\overline{\hat\mu(\Gamma)} \subset
\hat\mu(\Delta) \cup \overline{\hat\mu(\Gamma \setminus \Delta)}$, Proposition
\ref{kt-V} applies.
\end{proof}

{\bf Example 1.} {\it A strictly aperiodic continuous $\mu$ with 
$\lim_{n\to\infty} \|V^n(I-V)\|>0$.}

\noindent
Let $G=\mathbb T$, so $\Gamma=\mathbb Z$. Any continuous probability on $\mathbb T$
 is strictly aperiodic (e.g. \cite[Lemma 2.2]{CL}). Combining Theorems 5.2.2(a) and 
5.5.2(b) of \cite{Ru1}, we obtain that there exists a Cantor set $C \subset \mathbb T$ 
which is also a Kronecker set, such that every continuous probability $\mu$ supported 
on $C$ has $\overline{\hat\mu(\mathbb Z)} =\overline{\mathbb D}$. Existence of such
$\mu$ is by \cite[Remark 5.2.5]{Ru1}. When $U(t)$ is the regular representation in 
$L_2(G,dt)$,  $Vf=\mu*f$. Since $\{e_n(z):= z^n: n\in\mathbb Z\}$ is an orthonormal 
basis of eigenfunctions, $\sigma(V)=\overline{\hat\mu(\mathbb Z)}$. Hence
$\sigma(V)= \overline{\mathbb D}$, so by Proposition \ref{kt} 
$\lim_{n\to\infty} \|V^n(I-V)\|>0$.
\medskip

\begin{cor} Let $U(t)$ be a unitary representation of $G$, and let the probability $\mu$ 
satisfy, for some $C\ge 1$,
\begin{equation} \label{sf}
|1-\hat\mu(\gamma)|^2 \le C(1-|\hat\mu(\gamma)|^2) \quad 
\text{\rm for every }\ \gamma \in \Gamma.
\end{equation}
Then $V:=\int_G U(t)d\mu(t)$ satisfies $\|V^n(I-V)\| \to 0$.
\end{cor}
\begin{proof}
Fix $C \ge 1$. Then $\{\zeta \in \mathbb C: |1-\zeta|^2 \le C(1-|\zeta|^2)\}$ is precisely
the closed disk $D_C$ centered at $(1/(C+1),0)$ with radius $C/(C+1)$, by some computation.
Then $\overline{\hat\mu(\Gamma)}\cap \mathbb T \subset D_C \cap \mathbb T=\{1\}$, 
so $\|V^n(I-V)\| \to 0$.
\end{proof}

{\bf Remarks.} 1. Condition \eqref{sf} for $\mu$ on $\mathbb Z$ was used by Jones,
Ostrovskii and Rosenblatt in \cite[Theorem 1.8]{JOR} for proving a strong $L_2$-estimate
 for the square function of $V$ when $U(n)=U^n$, $U$ the unitary operator induced
by an invertible probability preserving transformation.

2.  When $\mu$ is adapted, \eqref{sf} implies strict aperiodicity.

3. The BAR condition \eqref{bar} implies \eqref{sf}

4.  The probability $\mu:=\frac12(\delta_0+\delta_1)$ on $\mathbb Z$ satisfies
\eqref{sf} and not \eqref{bar}.

5. If $\mu$ on $\mathbb Z$ is strictly aperiodic, then \eqref{sf} holds 
\cite[Theeorem 1.10]{JOR}.

6. In Example 1, $\mu$ does not satisfy \eqref{sf}, since $\lim_n \|V^n(I-V)\| > 0$.
\bigskip

{\bf Example 2.} {\it A strictly aperiodic $\mu$ on $\mathbb R$ which does not
satisfy \eqref{sf}.}

\noindent
Let $1, a, b, c \in \mathbb R$ be linearly independent over $\mathbb Q$. Define
$\mu:=\frac13(\delta_a+\delta_b+\delta_c)$. Then 
$\hat\mu(s)=\frac13(\e^{2\pi isa}+\e^{2\pi isb}+\e^{2\pi isc})$ for $s \in \mathbb R$.
 Fix $s$ such that $|\hat\mu(s)|=1$. By uniform convexity, 
$\e^{2\pi isa}=\e^{2\pi isb}=\e^{2\pi isc}$. If $s \ne 0$, then there exist 
$\theta \in[0,1)$ and distinct integers $n_1,n_2,n_3$ such that
$$
as= n_1+ \theta, \quad bs= n_2+ \theta, \quad cs= n_3+ \theta.
$$
Then $(b-a)s=(n_2-n_1)$ and $(c-a)s=(n_3-n_1)$. Equating the values of $s$ we
obtain the equality $a(n_3-n_2) +b(n_1-n_3)+c(n_2-n_1)=0$, contradicting the linear
independence. Thus $|\hat\mu(s)|=1$ implies $s=0$, so $\mu$ is strictly aperiodic.
\smallskip

We now prove that $\overline{\hat\mu(\mathbb R)}\cap \mathbb T=\mathbb T$.
Fix $z = \e^{2\pi is_0} \in \mathbb T$. By Kronecker's density theorem 
\cite[Theorem 443]{HW}, for $\epsilon>0$  there is a positive integer $n$ such that
$$
\max\{| {\rm e}^{2\pi ina} -z|, | {\rm e}^{2\pi inb} -z|, | {\rm e}^{2\pi inc} -z|\}
<\epsilon.
$$
hence $|\hat\mu(n)-z|\le
 \frac13(| {\rm e}^{2\pi ina} -z|+| {\rm e}^{2\pi inb} -z|+| {\rm e}^{2\pi inc} -z| ) <
 \epsilon$,
so $z\in \overline{\hat\mu(\mathbb R)}$. This shows that $\hat\mu(\mathbb R)$
is not contained in any disk centered at $(\rho,0)$ with radius $1-\rho$, $\rho>0$,
so \eqref{sf} fails.
\medskip

{\bf Remark.} Let $U(t)=\e^{itB}$ be the representation of $\mathbb R$ induced on
$L_2(m)$ by an ergodic measure preserving action of $\mathbb R$ on a $\sigma$-finite
 measure space $(\Omega,m)$. Then the support of the spectral measure of $B$
is $\mathbb R$ \cite[Corollaire 2]{Fo}, and for $V=\int_\mathbb R U(t)d\mu(t)$ we 
have $\sigma(V)=\overline{\hat\mu(\mathbb R)}$ (see Section \ref{R}). Hence for
$\mu$ of Example 2, $\sigma(V)\cap \mathbb T=\mathbb T$, so
$\lim_n \|V^n(I-V)\| >0$.
\bigskip

\section{ Spectra of averages of unitrary representations } \label{spectra}

%

\subsection{Spectra of averages}

We present a different proof of the general Theorem \ref{jrt} below, which
 was proved originally in \cite[Proposition 1.9]{JRT}. 

\begin{thm} \label{jrt}
Let $G$ be a LCA group with dual $\Gamma$, and let $U(t), t\in G$, be a weakly continuous
unitary representation of $G$ on a complex Hilbert space $H$. For a regular probability $\mu$
on $G$ define $V:= \int_G U(t) d\mu(t)$ (which is defined in the strong operator topology
of $H$). Then $\sigma(V) \subset \overline{\{\hat\mu(\gamma): \gamma\in \Gamma\}}$.
\end{thm}
\begin{proof} By the generalized Stone spectral theorem for unitary representations of
 LCA groups (see \cite{Am}), there exists a "regular" projection-valued spectral measure 
$E(\cdot)$ on the Borel subsets of $\Gamma$, such that 
\begin{equation} \label{ambrose}
U(t)=\int_\Gamma \gamma(t)E(d\gamma).
\end{equation}
By definition, $m_x(\cdot):=\langle E(\cdot)x,x \rangle$ is a regular positive finite
measure, and  $m_{x,y}(\cdot):=\langle E(\cdot)x,y \rangle$ is a regular signed measure.
We need only the well-defined weak spectral representation
$$
\langle U(t)x,y\rangle= \int_\Gamma \gamma(t) \langle E(d\gamma),y\rangle =
\int_\Gamma \gamma(t) m_{x,y}(d\gamma), \quad x,y \in H.
$$
Then $V$ is well-defined weakly, and  for any  $x,y \in H$ we have
$$
\langle V x,y \rangle = \int_G \langle U(t)x,y \rangle d\mu(t)=
	\int_G \big[\int_\Gamma \gamma(t) m_{x,y}(d\gamma)\big] d\mu(t).
$$
By Fubini's theorem, 
\begin{equation} \label{fubini}
\langle V x,y \rangle = \int_\Gamma \big[\int_G \gamma(t)d\mu(t) \big]  m_{x,y}(d\gamma)
=\int_\Gamma \hat\mu(\gamma) m_{x,y}(d\gamma).
\end{equation}
Since $\hat\mu: \Gamma\mapsto \overline{\mathbb D}$ is continuous, we can define
$E'(A)=E(\hat\mu^{-1}A)$ for Borel sets $A\subset \overline{\mathbb D}$. Then $E'(\cdot)$
is a spectral measure on the Borel sets of $\overline{\mathbb D}$. The change of variables formula
then yields
\begin{equation} \label{change-of-var}
\langle V x,y \rangle =\int_\Gamma \hat\mu(\gamma) m_{x,y}(d\gamma)=
\int_{\overline{\mathbb D}} \lambda \langle E'(d\lambda)x,y \rangle =
\int_{\overline{\hat\mu(\Gamma)}} \lambda \langle E'(d\lambda)x,y \rangle,
\end{equation}
the last equality since $E'(\overline{\hat\mu(\Gamma)})=E(\Gamma)=I$. Hence
$$
V=\int_{\overline{\hat\mu(\Gamma)}} \lambda E'(d\lambda).
$$
By uniqueness of the spectral measure of the normal operator $V$, $E'$ is its spectral 
measure, supported on the spectrum $\sigma(V)$, so 
$\sigma(V) \subset \overline{\hat\mu(\Gamma)}$.
\end{proof}

{\bf Remarks.} 1. Our proof does not require separability nor metrizability of $G$, which
are assumed in \cite{JRT}.

2. A different proof of Theorem \ref{jrt}, also based on the general spectral theorem,
but using the operational calculus \cite[Theorem 2(b)]{Am}, was presented in the proof
of \cite[Proposition 5.3]{CCL}. The advantage of the present proof is that it shows
the spectral measure of $V$ as $E'(\cdot)=E(\hat\mu^{-1}(\cdot))$, which is used  below.

3. A bounded continuous representation of a LCA $G$ by operators on $H$ is similar 
to a unitary representation \cite[Th\'eor\`eme 6]{Di}.

4. Theorem \ref{jrt} is improved in Corollary \ref{wsmt} below.

\begin{cor}[{\cite{CCL}}] \label{H-ritt}
The normal contraction $V$ of Theorem \ref{jrt} is Ritt if  there exists $C>0$ such that
\begin{equation} \label{BAR}
|1-\hat\mu(\gamma)| \le C(1-|\hat\mu(\gamma)|) \quad \forall \gamma \in \Gamma.
\end{equation}
\end{cor}
\begin{proof} The BAR condition \eqref{BAR} with Theorem \ref{jrt} imply that 
$\sigma(V)$ is contained in a closed Stolz region; since $V$ is normal, it is Ritt by 
Proposition \ref{stolz}.
\end{proof}

{\bf Remark.} Lancien and Le Merdy \cite[Theorem 6]{LLM} proved Corollary \ref{H-ritt}
for bounded continuous representations of $G$ on UMD Banach lattices, in particular on
$L^p$ spaces, $1<p<\infty$.  They do not consider the spectrum in general.
\medskip

{\bf Definition.} The {\it support} of a projection valued spectral measure $E(\cdot)$
defined on the Borel sets of $\Gamma$, denoted $\Sigma(E)$, is the set of all points  
$\gamma \in \Gamma$ such that $E(O) \ne 0$ for every open set $O$ containing 
$\gamma$. It is easily shown that $\Sigma(E)$ is closed, and its complement is the
union of all open sets $O$ with $E(O)=0$. Hence $\Gamma \setminus \Sigma(E)$ is
the maximal open set $O$ with $E(O)=0$. Consequently, every closed set $\Delta$
with $E(\Delta)=I$ contains $\Sigma(E)$.


\begin{prop} \label{sigma-E}
Let $U(t)$, $\mu$ and $V$ be as in Theorem \ref{jrt}, and $E(\cdot)$ the spectral 
measure of $U(t)$. Let $\Sigma(E) \subset \Gamma$ be the (closed) support of $E(\cdot)$.
If $\mu$ is adapted, then $\sigma(V) = \overline{\hat\mu(\Sigma(E))}:=
\overline{\{\hat\mu(\gamma): \gamma\in \Sigma(E)\}}$.
\end{prop}
\begin{proof}  Put $\sigma=\sigma(V)$; since $E'=E\circ \hat\mu^{-1}$ is the spectral 
measure of $V$, $E(\hat\mu^{-1}\sigma)=E'(\sigma)=I$.  Since $\sigma$ is closed 
and $\hat\mu$ is continuous, $\hat\mu^{-1}\sigma \supset \Sigma(E)$, and then  
$\sigma \supset \overline{\hat\mu(\Sigma(E))}$.

Since $\mu$ is adapted, $\hat\mu(\gamma)=1$ only when $\gamma=1$,
which implies the $\hat\mu$ is one-to-one. Hence
 $E'(\overline{\hat\mu(\Sigma(E))}) \supset E(\Sigma(E)) =I$.
Then, by \eqref{change-of-var} in the proof of  Theorem \ref{jrt}, we actually have 
$$
\langle Vx,y\rangle =\int_{\Sigma(E)} \hat\mu(\gamma)m_{x,y}(d\gamma) =
\int_{\overline{\hat\mu(\Sigma(E))}} \lambda \langle (E'(d\lambda)x,y\rangle.
$$
Since the support of the spectral measure of $V$ is $\sigma$, we have
$\sigma \subset \overline{\hat\mu(\Sigma(E))}$.

We conclude that $\sigma= \overline{\hat\mu(\Sigma(E))}$.
\end{proof}

{\bf Remark.} When $\mu$ is not adapted, we can look at the restriction of the representation
 to the closed subgroup $G_1$ generated by the support of $\mu$. Then $\mu$ is adapted
on $G_1$ and Proposition \ref{sigma-E} can be applied.

\smallskip

When $\Sigma(E)=\Gamma$, we can identify $\sigma(V)$ for every $\mu$, even if it is not 
adapted.

\begin{prop} \label{sigma-V}
Let $U(t)$, $\mu$ and $V$ be as in Theorem \ref{jrt}, and $E(\cdot)$ the spectral 
measure of $U(t)$. If $E(O) \ne 0$ for any non-empty open $O \subset \Gamma$, then 
 $\sigma(V) =\overline{\hat\mu(\Gamma)}=
\overline{\{\hat\mu(\gamma): \gamma\in \Gamma\}}$.
\end{prop}
\begin{proof} Put $\sigma=\sigma(V)$, and let $A=\overline{\hat\mu(\Gamma)}\setminus \sigma$.
Since $E'=E\circ \hat\mu^{-1}$ is the spectral measure of $V$, $E'(\sigma)=I$; 
as shown above, $E'(\overline{\hat\mu(\Gamma))}=I$. 
By Theorem \ref{jrt} $\sigma \subset \overline{\hat\mu(\Gamma)}$, so 
$0=E'(A)= E(\Gamma \setminus \hat\mu^{-1} \sigma)$. Since $\sigma$ is closed, 
$\Gamma \setminus \hat\mu^{-1} \sigma$ is open in $\Gamma$. The assumption yields
$\hat\mu^{-1} \sigma =\Gamma$, so $\hat\mu(\Gamma) \subset \sigma$. Since $\sigma$ is 
closed,  also $\overline{\hat\mu(\Gamma)} \subset \sigma$.
\end{proof}

\begin{cor}
Under all the assumptions of Proposition \ref{sigma-V}, if  $\mu$ is adapted and
$\|V^n(I-V)\| \to 0$, then $\mu$ is strictly aperiodic.
\end{cor}
\begin{proof} By Proposition \ref{kt} $\sigma(V)\cap \mathbb T=\{1\}$. Since every 
$\hat\mu(\gamma)$ is in $\sigma(V)$ by Proposition \ref{sigma-V},
$|\hat\mu(\gamma)|=1$ implies $\hat\mu(\gamma)=1$, and since $\mu$ is adapted
$\gamma=1$. Thus $|\hat\mu(\gamma)|<1$ for $\gamma \ne 1$.
\end{proof}

\begin{cor}
Let $U(t)$, $\mu$ and $V$ be as in Theorem \ref{jrt}. Assume that $E(\cdot)$, the spectral 
measure of $U(\cdot)$, satisfies $E(O) \ne 0$ for any non-empty open $O \subset \Gamma$. 
Then $V$ is Ritt if and only if $\mu$ satisfies the BAR condition \eqref{BAR}.
\end{cor}
\begin{proof}
The "if" part is Corollary \ref{H-ritt}.

When $V$ is Ritt, its spectrum is contained in a Stolz region $\Delta$. Then there is a $C$ 
such that $|1-\zeta| \le C(1-|\zeta|)$  for every $\zeta \in \Delta$. By Proposition 
\ref{sigma-V} $\hat\mu(\Gamma) \subset \sigma(V) \subset \Delta$, which yields \eqref{BAR}.
\end{proof}
\medskip

\subsection{Spectral mapping theorems}

Let $E(\cdot)$ be a regular spectral measure on the Borel sets of $\Gamma$. Then for 
every bounded measurable function $f: \Gamma \longrightarrow \mathbb C$ a bounded 
linear operator $\int_\Gamma f\,dE$ can be defined by 
$$
\langle \int_\Gamma f\,dE\,x,y \rangle :=
\int_\Gamma f(\gamma)\langle E(d\gamma)x,y \rangle, \qquad x,y \in H.
$$
This is extended in \cite[Theorem 13.24]{Ru} to unbounded functions: for every
measurable  $f: \Gamma \longrightarrow \mathbb C$ there is a densely defined closed
operator $\Psi (f)$ such that
$$
\langle \Psi (f)x,y \rangle = 
\int_\Gamma f(\gamma)\langle E(d\gamma)x,y \rangle, \qquad
 x\in\mathcal D(\Psi(f)),\quad y \in H.
$$
For $x\in\mathcal D(\Psi(f))$ we have $\|\Psi(f)x\|^2=
\int_\Gamma|f(\gamma)|^2\langle E(d\gamma)x,x \rangle$.
By \cite[Theorem 13.25]{Ru}, $\Psi(f)$ is bounded if and only if $f$ is bounded; in
that case $\Psi(f)^*=\psi(\bar f)$, and $\Psi(f)$ is normal.
\smallskip

Let $E(\cdot)$ be the spectral measure of a unitary representation $U(t)$.
Fix $t \in G$ and put $f_t(\gamma)=\gamma(t)$. Then, using \eqref{ambrose},
$\Psi(f_t)x=\int_\Gamma  \gamma(t) E(d\gamma)x= U(t)x$.

Conversely, given a regular spectral measure on $\Gamma$, the family $U(t):=\Psi(f_t)$,
where $f_t(\gamma):=\gamma(t)$, is a strongly continuous representation of $G$
(since $\Psi(fg)=\Psi(f)\Psi(g)$ for bounded $f$ and $g$ \cite[p. 345]{Ru}).

\begin{prop} \label{measurable}
Let $E(\cdot)$ be a regular spectral measure on $\Gamma$, and let 
$f: \Gamma \longrightarrow \mathbb C$ be measurable.
Then $\sigma(\Psi(f)) \subset \overline{f(\Sigma(E))}$.
\end{prop}
\begin{proof}
By the definitions in \cite[p. 303]{Ru}, for $f: \mathbb R \longrightarrow \mathbb C$
measurable, {\it the essential range of $f$ with respect to $E(\cdot)$}, which we
denote E-range($f$), is

\centerline{ E-range($f$)$=\{ \zeta\in\mathbb C: E(f^{-1}(O)) \ne 0 \  \text{\rm for
every open } O \ni \zeta\}.$}
\medskip

It is proved in \cite[Theoorem 13.27(c)]{Ru} (without any continuity or boundedness
assumption) that $\sigma(\Psi(f))= \text{\rm E-range}(f)$, so we prove that  
 $\text{\rm E-range}(f) \subset \overline{f(\Sigma(E))}$.
\medskip

Fix $\zeta \in \text{\rm E-range}(f)$. Let $O_n$ be the open ball in $\mathbb C$ of 
radius $1/n$ centered at $\zeta$. Since $E(f^{-1}(O_n)) \ne 0$ and $E(\cdot)$ is
supported on $\Sigma(E)$, $f^{-1}(O_n)\cap\Sigma(E) \ne\emptyset$, so contains a 
$\lambda_n \in \Sigma(E)$ with $f(\lambda_n) \in O_n$. Hence 
$\zeta \in \overline{f(\Sigma(E))}$.
\end{proof}

\begin{thm} \label{continuous}
Let $E(\cdot)$ be a regular spectral measure on $\Gamma$, and let 
$f: \Gamma \longrightarrow \mathbb C$ be continuous.
Then $\sigma(\Psi(f)) =\overline{f(\Sigma(E))}$.
\end{thm}
\begin{proof}
The inclusion $\sigma(\Psi(f)) \subset\overline{f(\Sigma(E))}$ is in Proposition 
\ref{measurable}.
\smallskip

For the opposite inclusion, we adapt arguments of \cite{MSE}.

Fix $\zeta \in \overline{f(\Sigma(E))}$, and let $O$ be an open set containing $\zeta$.
Then there exists $\lambda \in \Sigma(E)$ with  $f(\lambda) \in O$. Hence 
 $f^{-1}(O)$ is a non-empty open set; if $E(f^{-1}(O))=0$, 
then $f^{-1}(O)\cap \Sigma(E)=\emptyset$, since the support of $E(\cdot)$ is
$\Sigma(E)$, contradicting $\lambda \in  f^{-1}(O)$, so $E(f^{-1}(O)) \ne 0$. This for 
any open $O$ containing $\zeta$, so $\zeta \in \text{\rm E-range}(f)=\sigma(\Psi(f))$.
\end{proof}

 Let $E(\cdot)$ be the spectral measure of a unitary representation $U(t)$.
For a bounded complex Borel measure $\nu$ on $G$, the operator
$U(\nu):=\int_G U(t)d\nu(t)$ is a well-defined (at least weakly)  bounded operator.
We then have the following "spectral mapping theorem", which improves Theorem \ref{jrt} 
and extends Proposition \ref{sigma-E}.

 \begin{cor}[{\bf spectral mapping theorem}] \label{wsmt}
Let $U(t)$ be a continuous unitary representation of $G$, with spectral measure $E(\cdot)$.
Then for every bounded complex measure $\nu$, 
$\ \sigma(\int_G U(t)d\nu(t))=\overline{\hat\nu(\Sigma(E))}$. In particular, for
$\nu=\delta_t$ we obtain  $\sigma(U(t))= \overline{\{\gamma(t): \gamma\in \Sigma(E)\}}$.
\end{cor}
\begin{proof}
The bounded function $\hat\nu : \Gamma \longrightarrow \mathbb C$ is continuous 
\cite[Theorem 1.3.3(a)]{Ru1}, and, by \eqref{fubini} (with $\mu$ replaced by $\nu$), 
\begin{equation} \label{Psi}
\Psi(\hat\nu)= \int_\Gamma \hat\nu(\gamma) E(d\gamma)= \int_G U(t)d\nu(t).
\end{equation}
The assertion follows from Theorem \ref{continuous} applied to $f=\hat\nu$.
\end{proof} 

\begin{cor} \label{connes0}
$|\hat\nu(\gamma)| \le \|U(\nu)\|$ for every $\gamma \in \Sigma(E)$ and  bounded measure $\nu$. 

In particular, $U(\nu)=0$ implies $\hat\nu(\gamma)=0$ for every $\gamma \in \Sigma(E)$.
\end{cor}
\begin{proof} 
By Corollary \ref{wsmt}, $\sigma(U(\nu))=\overline{\hat\nu(\Sigma(E))}$. Hence for 
 every $\gamma \in \Sigma(E)$ we have 
$|\hat\nu(\gamma)| \le \rho(U(\nu))\le \|U(\nu)\|$.
\end{proof}
We fix a Haar measure $m_G$, and identify  the space $L_1(G):=L_1(G,m_G)$ with 
the space of bounded complex measures absolutely continuous with respect to  $m_G$
When $f \in L_1(G)$, and $f =d\nu/dm_G$, we denote  $U(\nu)$ by $U(f)$.

\begin{prop} \label{Sigma=Gamma}
Let $U(t)$ be  a unitary representation of $G$, with spectral measure $E(\cdot)$.
Then the following are equivalent:

(i) $\Sigma(E)=\Gamma$.

(ii) For every bounded complex measure $\nu \ne 0$ we have $\int_G U(t)d\nu(t) \ne 0$.

(iii) For every $0 \ne f \in L_1(G,m_G)$ we have $\int_G U(t)f(t)dm_G(t) \ne 0$.
\end{prop}
\begin{proof} (i) implies (ii): If $U(\nu)=0$, then, by Corollary \ref{connes0} and (i),
$\hat\nu(\gamma)=0$ for every $\gamma \in \Gamma$. Hence $\nu=0$ by the 
uniqueness theorem \cite[1.7.3(b)]{Ru1}.
\smallskip

Obviously (ii) implies (iii). 

Assume (iii). If $\Sigma(E) \ne \Gamma$, then, by applying \cite[Theorem 1.6.4]{Ru1} 
to $\Gamma\setminus \Sigma(E)$ we obtain $0\ne f \in L_1(G,m_G)$ such that 
$\hat f (\gamma)=0$ for every $\gamma\in\Sigma(E)$. Hence, by \eqref{Psi},
 $$
\int_G U(t)f(t)dm_G(t)= \int_\Gamma \hat f(\gamma)E(d\gamma) =
\int_{\Sigma(E)} \hat f(\gamma)E(d\gamma) = 0,
$$
contradicting (iii). Consequently $\Sigma(E)=\Gamma$.
\end{proof}

{\bf Definition.} A unitary representation $(U(t))_{t \in G}$ is called {\it aperiodic} if
$U(t)\ne I$ for $t \ne 0$.

\begin{prop} \label{example1}
Assume $G$ has a countable base.
Let $(\theta(t))_{t\in G}$ be  an ergodic action by measure preserving transformations 
of a probability space $(\Omega,m)$, with $\mathbf{U}:=(U(t))_{t \in G}$ the 
unitary representation induced on $L_2(\Omega,m)$. If $\mathbf{U}$ is aperiodic, 
then the support of its spectral measure is $\Gamma$.
\end{prop}
\begin{proof} Let $E(\cdot)$ be the spectral measure of $\mathbf{U}$, with (closed)
support $\Sigma(E)$. Put $G_0:=\{t \in G: \gamma(t)=1,\ \forall\ \gamma\in\Sigma(E)\}$. 
Clearly $G_0$ is a closed subgroup of $G$. Fix $t \in G_0$; then
 $\sigma(U(t))= \overline{\{\gamma(t): \gamma\in \Sigma(E)\}}$ by Corollary 
\ref{wsmt}, so $\sigma(U(t))=\{1\}$. Since $U(t)$ is normal, $U(t)=I$, and by
aperiodicity $t=0$. Hence $G_0=\{0\}$.

By \cite[Proposition 5]{Fo},  $\Sigma(E)$ is a closed subgroup of $\Gamma$.
 By \cite[Theorem 2.1.2]{Ru}, the dual of $\Gamma/\Sigma(E)$ is isomorphically
homeomorphic to $G_0$, which yields $\Sigma(E)=\Gamma$.
\end{proof}
\medskip

\subsection{The Arveson spectrum} \label{Arveson}

Let $\mathbf{T}:=\{T(t): t\in G\}$ be a (bounded) weakly continuous representation
of the LCA group $G$ in a Banach space $X$.  For any bounded complex Borel measure
$\nu$ on $G$, the operator $T(\nu):=\int_G T(t)d\nu(t)$ is well-defined (weakly) on $X$.
\smallskip

We fix a Haar measure $m_G$. 
For a bounded complex measure $\nu$ absolutely continuous with respect to  $m_G$,
 we  may denote $T(\nu)$ by $T(f)$, where $f =d\nu/dm_G$.
\smallskip

{\bf Definition \cite{Ar1}.} The {\it Arveson spectrum} of $\mathbf{T}$, denoted 
$Sp(\mathbf{T})$, is 
\begin{equation} \label{arveson}
Sp(\mathbf{T}):= \{\gamma \in \Gamma: \hat f(\gamma)=0 \quad 
\text{\rm whenever } f\in L_1(G)\  and \ T(f)=0\}.
\end{equation}
It is easily shown that $Sp(\mathbf{T})$ is a closed set.

\begin{prop} \label{inclusion}
Let $\mathbf{U}=U(t)_{t\in G}$ be a unitary representation of $G$ in a complex Hilbert
space, with spectral measure $E(\cdot)$. Then $\Sigma(E)$,  the support of $E(\cdot)$,
is a subset of the Arveson spectrum. Hence $Sp(\mathbf{U}) \ne \emptyset$.
\end{prop}
\begin{proof} Let $U(\nu)=0$ for a measure with $|\nu| \ll m_G$. By Theorem
\ref{continuous} and \eqref{Psi}, $ \overline{\hat\nu(\Sigma(E))}=\sigma(U(\nu))=
\{0\}$. Hence $\hat\nu(\gamma)=0$ for every $\gamma \in\Sigma(E)$. This
implies that $\Sigma(E) \subset Sp(\mathbf{U})$.
\end{proof}

\begin{thm} \label{equality}
Let $\mathbf{U}=U(t)_{t\in G}$ be a unitary representation of $G$ in a complex Hilbert 
space, with spectral measure $E(\cdot)$. Then $Sp(\mathbf{U})=\Sigma(E)$.
\end{thm}
\begin{proof} By Proposition \ref{inclusion}, $\Sigma(E) \subset Sp(\mathbf{U})$.
If there is no equality, there is a $\gamma_0 \in Sp(\mathbf{U})\setminus \Sigma(E)$.
Since $\Sigma(E)$ is closed, there exists an open set $O$ with $\gamma_0 \in O$
and $O\cap \Sigma(E)= \emptyset$. By \cite[Theorem 2.6.2]{Ru1}, there exists
$f \in L_1(G)$ with $\hat f(\gamma_0)=1$ and $\hat f(\gamma)=0$ outside $O$,
in particular on $\Sigma(E)$. But then by \eqref{Psi},
 $U(f)=\int_G U(t)f(t)dm_G(t)=\int_{\Sigma(E)}\hat f(\gamma) E(d\gamma)=0$, 
so $\hat f(\gamma_0)=0$ since $\gamma_0 \in Sp(\mathbf{U})$, a contradiction. 
Hence $Sp(\mathbf{U})=\Sigma(E)$.
\end{proof}

{\bf Remarks.} 1. St\o rmer \cite[p. 208]{St} stated the equality 
$\Sigma(E)=Sp(\mathbf{U})$ without proof. We have not found any published proof.

2. Proposition \ref{inclusion} can be deduced from \cite{Ar1}.
For $x \in H$, Arveson \cite[Definition 2.1]{Ar1} defined 
$$
Sp(\mathbf U)(x) := \{\gamma \in \Gamma: \hat f(\gamma)=0 \quad 
\text{\rm whenever} \  f \in L_1(G)\  and \ U(f)x=0\}.
$$
It is easy to show that $Sp(\mathbf U)(x) \subset Sp(\mathbf U)$, so "the spectral
subspace" $M^U(Sp(\mathbf{U}))$ as defined in \cite{Ar1} is $H$. As observed in
\cite[p. 226]{Ar1}, and proved in \cite[Proposition 8.3.2]{Pe}, for 
$\Delta \subset \Gamma$ closed, $M^U(\Delta)=E(\Delta)H$. Hence 
$E(Sp(\mathbf{U}))=I$, so $\Sigma(E) \subset Sp(\mathbf{U})$.
\smallskip

{\bf Note.} By \cite[Corollary 7.2.5(c)]{Ru}, $Sp(\mathbf{U})(x)=\emptyset$ implies
that $U(f)x=0$ for every $f \in L_1(G)$, and  $x=0$ by continuity of $\mathbf{U}$.
Hence $Sp(\mathbf{U})(x)$ is not empty for $x \ne 0$.

\begin{cor} \label{Uf=0}
Let $\mathbf{U}$ be as in Proposition \ref{inclusion}. The following conditions are
equivalent for $f \in L_1(G)$:

(i) $\int_G U(t)f(t)dm_G(t)=0$.

(ii) $\hat f(\gamma)=0 \quad \forall\ \gamma \in Sp(\mathbf{U})$.

(iii) $\hat f(\gamma)=0 \quad \forall\ \gamma \in \Sigma(E)$.
\end{cor}
\begin{proof} $(i) \Rightarrow (ii)$ by definition of $Sp(\mathbf{U})$.
$(ii) \Rightarrow (iii)$ by Proposition \ref{inclusion}, and $(iii) \Rightarrow (i)$ since
$\int_G U(t)f(t)dm_G(t)= \int_{\Sigma(E)}\hat f(\gamma)E(d\gamma)$ by \eqref{Psi}.
\end{proof}

{\bf Remark.} We could not prove directly, without referring to Arveson's spectrum, 
that (i) implies (iii) in Cororllary \ref{Uf=0}.

\begin{thm}[{\bf spectral mapping theorem}] \label{WSMT}
Let $\mathbf{U}=U(t)_{t\in G}$ be a unitary representation of $G$ in a complex 
Hilbert space. 
For any bounded complex Borel measure $\nu$  we have  
$\sigma(U(\nu)) =\overline{\hat\nu(Sp(\mathbf{U}))}$.
\end{thm}
\begin{proof} Let $E(\cdot)$ be the  spectral measure of $\mathbf{U}$.
	By Corollary \ref{wsmt} and Theorem \ref{equality},
$\sigma(U(\nu))= \overline{\hat\nu(\Sigma(E))} = \overline{\hat\nu(Sp(\mathbf{U}))}$.
\end{proof}

{\bf Remarks.} 1. D'Antoni et al. \cite[Lemma 1]{DLZ} proved, for any bounded 
continuous representation $\mathbf{T}$ in a Banach space, that for any $\nu \in M(G)$,
$\overline{\hat\nu(Sp(\mathbf{T}))} \subset \sigma(T(\nu))$, where, as beofre,
$T(\nu)x=\int_G T(t)xd\nu(t)$. See also \cite[Proposition 2.10]{Ar2}. 

2. For a representation $\mathbf{T}$ in a Banach space as above, the equality 
$\sigma(T(\nu)) = \overline{\hat\nu(Sp(\mathbf{T}))}$ is proved in 
\cite[Theorem 1]{DLZ} when the continuous part of $\nu$ is absolutely continuous 
(maybe zero, i.e. $\nu$ is discrete), but may fail in general. For unitary representations
in Hilbert space, the equality holds  for any $\nu$, by Theorem \ref{WSMT}.

3. The counter-example in \cite{DLZ} is in $L_1(G)$. Theorem 2.13 in \cite{CL} provides 
an example of a continuous non-singular probability $\mu$ on $\mathbb T$ such that 
for $\mathbf{T}$ the regular representation in the {\it reflexive} $L_p(\mathbb T)$,
 $1<p<2$, the equality $\sigma(T(\mu)) = \overline{\hat\mu(Sp(\mathbf{T}))}$ fails.

\begin{cor} \label{connes}
Let $\mathbf{U}:=U(t)_{t \in G}$ be a unitary representation. Then the following are
equivalent for $\gamma \in \Gamma$:

(i) $\gamma \in \Sigma(E)$.

(ii) $\gamma \in Sp(\mathbf{U})$.

(iii) $|\hat\nu(\gamma)| \le \|U(\nu)\|$ for every  bounded complex measure $\nu$. 

(iv) $|\hat f(\gamma)| \le \|U(f)\|$ for every  $f \in L_1(G)$.

(v) If $\nu$ is a bounded complex measure with $U(\nu)=0$, then $\hat\nu(\gamma)=0$.

(vi) If $f \in L_1(G)$ with $U(f)=0$, then $\hat f(\gamma)=0$.
\end{cor}
\begin{proof} (i) $\Leftrightarrow$ (ii) by Theorem \ref{equality}.
(i) implies (iii) by Corollary \ref{connes0}. Obviously (iii) implies (iv) and (v), and
either (iv) or (v) implies (vi). (vi) implies (i) by the definition of Arveson's spectrum.
\end{proof}

{\bf Remarks.} 1. The equivalence of (ii) -- (vi), due to Connes, holds for bounded
representations in Banach spaces \cite[Proposition 8.1.9]{Pe} (condition (v), which is 
between (iii) and (vi), is not mentioned there). Our proof for the unitary case is 
different, being based on results for the spectral measure. 

2. An additional equivalent condition, also due to Connes, which justifies the term 
"spectrum", is given in \cite[Proposition 8.1.9(iii)]{Pe}.

\begin{prop} \label{SpU=Gamma}
Let $U(t)$ be  a unitary representation of $G$, with spectral measure $E(\cdot)$.
Then the following are equivalent:

(i) $Sp(\mathbf{U})=\Gamma$.

(ii) For every bounded complex measure $\nu \ne 0$ we have $\int_G U(t)d\nu(t) \ne 0$.

(iii) For every $0 \ne f \in L_1(G,m_G)$ we have $\int_G U(t)f(t)dm_G(t) \ne 0$.

(iv) $\Sigma(E)=\Gamma$.
\end{prop}
\begin{proof} Combine Proposition \ref{Sigma=Gamma} with Theorem \ref{equality}.
\end{proof}

{\bf Remark.} Proposition \ref{example1} provides an example where $Sp(\mathbf{U})=\Gamma$.
\medskip

For the sake of completeness, we state the next result, proved in 
\cite[p. 217, Corollary 4]{Ar2} for bounded representations in Banach spaces.

\begin{prop}
Let $\mathbf{U}$ be a unitary representation. If a bounded measure $\nu$ has
$\hat\nu(\gamma) \ne 0$ for every $\gamma \in Sp(\mathbf{U})$, then $U(\nu)$ is 
one-to-one and has a dense range.
\end{prop}
\bigskip

\section{Unitary representations of $\mathbb Z$}

In this section we identify precisely the spectra of averages of unitary representations of
$\mathbb Z$, which are averages of the powers of a unitary operator. 

\begin{prop} \label{unitary}
Let $U$ be a unitary operator on a complex Hilbert space. Then for every probability 
$\mu:=(p_k)_{k \in \mathbb Z}$ on $\mathbb Z$, 
$$
\sigma\Big(\sum_{k\in\mathbb Z} p_k U^k\Big) =\hat\mu(\sigma(U)):=
\Big\{\sum_{k\in\mathbb Z} p_k \lambda^k: \lambda \in \sigma(U)\Big\}.
$$
\end{prop}
\begin{proof} Let $E(\cdot)$ be the spectral measure of $U$; its support is $\sigma(U)$,
by \cite[Lemma X.2.3(i)]{DS}.
 The function $\phi(\zeta):=\sum_{k\in\mathbb Z} p_k\zeta^k$ is continuous on 
the closed disk $\overline{\mathbb D}$, and we denote by $f$ its restriction to
$\sigma(U) \subset \mathbb T$. Then $V:= \sum_k p_k U^k$ is $f(U)$. 
Note that $V$ is a normal contraction, and $\sigma(V) \subset \overline{\mathbb D}$.
By \cite[Corollary X.2.9(iii)]{DS}, applied to $\sigma(U)$ with its Borel sets,
$$
\sigma(V)=\bigcap_{\{\Delta: E(\Delta)=I\}} \overline{f(\Delta)}.
$$
Since $E(\sigma(U))=I$ and $f(\sigma(U))$ is compact, we obtain 
$\sigma(V) \subset f(\sigma(U))=\hat\mu(\sigma(U))$.

Denote by $E_V(\cdot)$  the spectral measure of $V$, which is supported on $\sigma(V)$.
Now let $ \zeta \notin \sigma(V)$. Since $\sigma(V)$ is closed, there exists an open set 
$\Delta \subset \overline{\mathbb D}$ such that $\zeta \in \Delta$ and 
$\Delta\cap\sigma(V)=\emptyset$. Then $E_V(\Delta)=0$. By continuity of $f$,
$f^{-1}(\Delta)$ is open. By \cite[Corollary X.2.10]{DS}, $E_V(\Delta)=E(f^{-1}(\Delta))$; 
hence $f^{-1}(\Delta)\cap \sigma(U)=\emptyset$, which implies that for any
$\lambda \in \sigma(U)$, $f(\lambda)\ne \zeta$; thus $\zeta \notin f(\sigma(U))$.
 This yields $f(\sigma(U)) \subset \sigma(V)$.
\end{proof}

{\bf Remarks.} 1. Another proof of of Proposition \ref{unitary} can be obtained by adapting 
Dungey's proof of \cite[Theorem 2.1]{D} to invertible operators  $T$ in Banach spaces
 with both $T$ and $T^{-1}$ power-bounded and measures in $L^1(\mathbb Z)$.

2. Since the support of the spectral measure of $U$ is $\sigma(U)$, we have 
$\Sigma(E)= \sigma(U)$. Then Proposition \ref{unitary} follows from the general 
Corollary \ref{wsmt}.     

\begin{cor} \label{shamo}
Let $U$ be a unitary operator on $H$. Then the following are equivalent:

(i) $\sigma(U)= \mathbb T$.

(ii) For every probability $\mu=(p_k)_{k\in\mathbb Z}$ on $\mathbb Z$,
$\ \sigma(\sum_k p_k U^k) =\hat\mu(\mathbb T)$.

(iii) For every {\rm strictly aperiodic} probability $\mu=(p_k)_{k\in\mathbb Z}$ on 
$\mathbb Z$, $\ \sigma(\sum_k p_k U^k) =\hat\mu(\mathbb T)$.
\end{cor}
\begin{proof} (i) implies (ii) by Proposition \ref{unitary} (or by Proposition \ref{sigma-V}).

Obviously (ii) implies (iii).
Assume (iii). If $\sigma(U) \ne \mathbb T$, then there exists  
$\gamma \in\mathbb T$ which is not in $\sigma(U)$. Take
$\mu=\frac12 \delta_0 +\frac12\delta_1$, i.e. $p_0=p_1=\frac12$, other $p_k$ 
zero, and put $V= \frac12I+\frac12U$. Then $\mu$ is stricltly aperiodic, and 
$\hat\mu(\mathbb T)=\{z\in \mathbb C: |z-\frac12|=\frac12\}$. Direct computation
yields $\sigma(V)=\{z: z=\frac12(1+ \lambda),\, \lambda \in \sigma(U)\}$. Sinc
$\gamma \notin \sigma(U)$, $\frac12(1+\gamma) \notin \sigma(V)$, though it is in
$\hat\mu(\mathbb T)$. This contradicts (iii). Hence (iii) implies (i).
\end{proof}

{\bf Remarks.} 1.  The equivalence of (i) and (ii) in Corollary \ref{shamo} follows also
 from the general result in Proposition \ref{Sigma=Gamma}; the above proof,
including the additional condition (iii), is more direct.

2. A. Ionescu Tulcea (Bellow) \cite{IT} and independently Foia\c s 
\cite[Corollaire 1]{Fo} proved that the spectrum of the unitary operator induced on $L^2$
by an ergodic measure preserving transformation of a probability space is $\mathbb T$.

3. Arveson \cite{Ar} deduced from his results that a unitary $U$ has spectrum
$\mathbb T$ if and only if for every $n\ge 1$ there is a vector $x \ne 0$ such that
$x,Ux,\dots,U^nx$ are orthogonal.

\begin{cor} \label{sap}
Let $U$ be a unitary operator on a complex Hilbert space $H$. Let
 $\mu=(p_k)_{k \in \mathbb Z}$ be a strictly aperiodic probability on $\mathbb Z$; then 
$V:=\sum_k p_kU^k$ satisfies $\|V^n(I-V)\| \to 0$, and $V^n$ converges strongly.
\end{cor}
\begin{proof} We show that $\sigma(V)\cap\mathbb T \subset\{1\}$. Assume 
$\gamma \in\sigma(V)$ satisfies $|\gamma|=1$. Then, by Proposition \ref{unitary},
 $\gamma=\sum_k p_k \lambda^k$ for some $\lambda \in \sigma(U) \subset \mathbb T$. 
By uniform convexity of $\mathbb T$, we obtain that
$\lambda^k =\gamma$ for every $k$ in $\mathcal S$, the support of $\mu$.
Then $\lambda^{k-j}=1$ for $k,j \in \mathcal S$, so 
$\mathcal S-\mathcal S \subset \{n: \lambda^n=1\}$, which is a  group, therefore it
is $\mathbb Z$ by  strict aperiodicity; hence $\lambda=1$, so $\gamma=1$.
 By Proposition \ref{kt} $\|V^n(I-V)\| \to 0$. By the ergodic decomposition 
$H=\{x: Vx=x\}\oplus \overline{(I-V)H}$, which yields strong convergence of $V^n$.
\end{proof}

{\bf Remark.} Since every signed measure on $\mathbb Z$ is absolutely continuous with 
respect to the counting (Haar) measure, strict aperiodicity yields 
Corollary \ref{sap} also by Proposition \ref{ac}.
\medskip

For the sake of completeness we include the following theorem.

\begin{thm}
Let $\theta$ be an invertible measure preserving transformation of an atomless probability 
space $(\Omega,\mathcal B,m)$, and $U$ the induced unitary operator on $L^2(\Omega,m)$.
Let $\mu:=(p_k)_{k\in\mathbb Z}$ be an adapted probability on $\mathbb Z$ and put 
$V:=\sum_k p_k U^k$. Then the following are equivalent;

(i) $\sigma(V)$ is contained in a closed Stolz region.

(ii) $V$ is Ritt.

(iii) $\mu$ satisfies the bounded angular ratio condition \eqref{BAR}.

(iv) For every $f \in L^2(\Omega,m)$,  $\ V^nf$ converges a.e.

(v) For every $f \in L^p(\Omega,m)$, $1<p< \infty$, $\ V^nf$ converges a.e.
\end{thm}
\begin{proof} The equivalence of (i) and (ii) is in Proposition \ref{stolz}.

(iii) implies (ii) by Corollary \ref{H-ritt}. and (ii) implies (iv) by Theorem \ref{lmx}.

(iv) implies (iii) by Losert \cite{Lo}, and (iii) implies (v) by  Bellow et al. 
\cite[Theorem 1.6]{BJR1}
\end{proof}

{\bf Remark.} If $\mu$ is not adapted, its support generates a group $d\mathbb Z$,
and the theorem applies to $\theta^d$.
\medskip

\section{Unitary representations of $\mathbb R$} \label{R}


Let $B$ be a self-adjoint (possibly unbounded) operator on $H$, with $E(\cdot)$ its spectral
 measure, which is supported in $\sigma(B) \subset \mathbb R$ \cite[Theorem XII.2.3]{DS}.
For a a bounded measurable $f: \mathbb R \longrightarrow \mathbb C$, the operator
$f(B):=\int_\mathbb R f(\lambda) E(d\lambda)= \int_{\sigma(B)} f(\lambda) E(d\lambda)$
is a well-defined (at least weakly) bounded operator. See Lemma \ref{f(b)} below.

Given a continuous unitary representation $\{U(t)\}_{t \in\mathbb R}$, by Stone's theorem
\cite[Theorem XII.6.1]{DS} there exists a self-adjoint operator $B$, with spectral measure
$E(\cdot)$, such that $U(t)=\e^{itB}$; then by the definitions
\begin{equation} \label{stone}
U(t)= \int_{\sigma(B)}  \e^{it\lambda}E(d\lambda) =
\int_{\mathbb R} \e^{it\lambda}E(d\lambda).
\end{equation}

We need the following general lemma.

\begin{lem} \label{f(b)}
Let $B$ be a self-adjoint operator with spectral measure $E(\cdot)$, and let
 $f: \mathbb R \longrightarrow \mathbb C$ be bounded measurable. Then:

(i) $f(B) :=\int_{\sigma(B)} f(\lambda) E(d\lambda)$ is a bounded normal operator, with
 $\|f(B)\| \le \underset{\lambda\in\sigma(B)}\sup |f(\lambda)|$.

(ii) $\sigma(f(B)) \subset \overline{f(\sigma(B))}$.

(iiii) The spectral measure of $f(B)$, defined on the Borel sets of $\mathbb C$, is
$E_f(\Delta)=E(f^{-1}(\Delta))$, $\Delta \subset \mathbb C$ Borel.
\end{lem}
\begin{proof} (i): Boundedness of $f(B)$ is immediate. Normality follows from
 \cite[Corollary X.2.9]{DS}, since $f(B)^*=\overline{f}(B)$.

(ii): By  \cite[Theorem XII.2.9(b)]{DS},
$$
\sigma(f(B))=\bigcap_{\{\Delta: E(\Delta)=I\}} \overline{f(\Delta)}.
$$
Since $E(\sigma(B))=I$, we obtain $\sigma(f(B)) \subset \overline{f(\sigma(B))}$.

(iii): When $f$ is real valued, this is \cite[Theorem XII.2.9(c)]{DS}. We adapt its proof
to obtain the complex case. Define $E'(\Delta):=E(f^{-1}(\Delta)))$,
$\Delta \subset \mathbb C$ Borel; then $E'(\cdot)$ is a spectral measure on Borel sets of
$\mathbb C$. Let $f_n(\lambda)= f(\lambda)$ when $|f(\lambda)| \le n$ and zero otherwise.
Then change of variables yields
$$
\int_{\{\zeta \in\mathbb C: |\zeta| \le n\}} \zeta E'(d\zeta) =
\int_\mathbb C \zeta\,\mathbf{1}_{\{ |\zeta| \le n\}}(\zeta)  E'(d\zeta) =
$$
$$
\int_\mathbb R f(\lambda) \mathbf{1}_{\{|f(\lambda)| \le n\}}(\lambda) E(d\lambda)=
\int_\mathbb R f_n(\lambda) E(d\lambda)=f_n(B).
$$
Since $f$ is bounded, $f_n(B) \to f(B)$ (at least weakly) by the bounded convergence 
theorem, which yields $f(B)=\int_\mathbb C \zeta E'(d\zeta)$. By the uniqueness of 
the spectral measure of bounded normal operators, $E_f=E'$.
\end{proof}

In fact, for every measurable $f: \mathbb R \longrightarrow \mathbb C$, one can define
a densely defined closed operator $f(B)$ \cite[Definiton XII.2.5 and Theorem XII.2.6]{DS}
such that
$$
\langle f(B)x,y \rangle=\int_{\mathbb R} f(\lambda)\langle E(d\lambda)x,y \rangle ,
\qquad x \in\mathcal D(f(B)), \quad y \in H.
$$  
This is a special case of \cite[Theorem 13.24]{Ru}, so $f(B)=\Psi(f)$.
 
By \cite[Lemma XII.1.3]{DS} $\sigma(B)$ is closed, and by \cite[Theorem XII.2.9(b)]{DS} 
(with $f(t)=t$),  the support of $E(\cdot)$ is precisely $\sigma(B) \subset \mathbb R$, 
so $\Sigma(E)=\sigma(B)$. Hence, by Proposition \ref{measurable}, 
$\sigma(f(B)) \subset \overline{f(\sigma(B))}$, and Theorem \ref{continuous}
 yields the following "spectral mapping theorem".

\begin{prop} \label{unbounded}
Let $B$ be a self-adjoint operator with spectral measure $E(\cdot)$, and let
 $f: \mathbb R \longrightarrow \mathbb C$ be  continuous. Then
$\sigma(f(B))= \overline{f(\sigma(B))}$.
\end{prop}
\medskip

Recall that the dual of $\mathbb R$ is $\mathbb R$, with $\lambda \in \mathbb R$ 
defining the character $\chi_{_\lambda}(t):=\e^{it\lambda}, \ t\in\mathbb R$.
\smallskip

The following result can be obtained from Proposition \ref{unbounded}, whose proof 
depends on \cite[Theorem 13.27(c)]{Ru} (see proof of Theorem \ref{continuous} and
\cite{MSE}); it is a special case of Corollary \ref{wsmt}, since $\Sigma(E)=\sigma(B)$. 
We give an independent proof.

\begin{prop} \label{spectrum}
Let $\{U(t)=\e^{itB}: t\in \mathbb R\}$ be a continuous unitary representation of
$\mathbb R$.  Then for every probability $\mu$ on $\mathbb R$,
$ \sigma\big(\int_{\mathbb R} U(t)d\mu(t) \big) = \overline{\hat\mu(\sigma(B))}$.
\end{prop}
\begin{proof} By the representation \eqref{stone} and  the uniqueness of the spectral 
measure in the general Stone spectral  theorem \eqref{ambrose}, the spectral
measure of $B$ is that of \eqref{ambrose} for $U(t)$.

Fix a probability $\mu$ on $\mathbb R$; then  $V=\int_{\mathbb R} U(t)d\mu(t)$
 is a normal contraction.

For $\lambda \in \mathbb R$ define $f(\lambda):=\int_{\mathbb R} \e^{it\lambda}d\mu(t)$.
Then $f$ is bounded and continuous on $\mathbb R$. By definition of the bounded 
normal operator $f(B)$ in Lemma \ref{f(b)}(i) and Fubini's theorem, we obtain
$$
f(B)= \int_{\sigma(B)} f(\lambda)E(d\lambda)=
\int_{\sigma(B)} \Big[\int_{\mathbb R} \e^{it\lambda}d\mu(t) \Big] E(d\lambda)=
$$
$$
\int_{\mathbb R} \Big[\int_{\sigma(B)} \e^{it\lambda}E(d\lambda) \Big] d\mu(t)=
\int_{\mathbb R} U(t)d\mu(t):=V.
$$
By Lemma \ref{f(b)}(ii),
 $\sigma(V) =\sigma(f(B))  \subset \overline{f(\sigma(B))}= \overline{\hat\mu(\sigma(B))}$.

The proof that $f(\sigma(B)) \subset \sigma(V)$ is the same as the last paragraph of
the proof of Proposition \ref{unitary}, except that Lemma \ref{f(b)}(iii) is used instead
of \cite[Corollary X.2.10]{DS}. Since $\sigma(V)$ is compact, 
$\overline{f(\sigma(B))} \subset \sigma(V)$. Hence
$\sigma(V)=\overline{\hat\mu(\sigma(B))}$. 
\end{proof}

\begin{cor}
Let $\{U(t)=\e^{itB}: t\in \mathbb R\}$ be a continuous unitary representation of 
$\mathbb R$. If $\sigma(B)=\mathbb R$, then for every probability $\mu$ on $\mathbb R$,
$ \sigma\big(\int_{\mathbb R} U(t)d\mu(t) \big) = \overline{\hat\mu(\mathbb R)}$.
\end{cor}

\bigskip

For any bounded complex measure $\nu$ on $\mathbb R$, the operator 
$U(\nu):=\int_{\mathbb R} U(t)d\nu(t)$ is well-defined. 

\begin{thm}
Let $\{U(t)=\e^{itB}: t\in \mathbb R\}$ be a continuous unitary representation of 
$\mathbb R$. Then the following are equivalent:

(i) $\sigma(B)=\mathbb R$.

(ii) For every  bounded complex measure $\nu$ on $\mathbb R$, 
$\ \sigma(U(\nu))= \overline{\hat\nu(\mathbb R)}$.

(iii) For every bounded complex measure $\nu\ne 0$ on $\mathbb R$, $\ U(\nu) \ne 0$.

(iv) For every absolutely continuous bounded complex measure $\nu\ne 0$ on $\mathbb R$,
$\ U(\nu) \ne 0$.
\end{thm}
\begin{proof} (i) implies (ii) by adapting the proof of Proposition \ref{spectrum}.

Assume (ii).  If $U(\nu)=0$, then by (ii) $\overline{\hat\nu(\mathbb R)}=\{0\}$, which
yields $\hat\nu(s)=0$ for every $s \in\mathbb R$. By the uniqueness theorem $\nu=0$. 
Thus (iii) holds, and obviously implies (iv).

Assume (iv). By Proposition \ref{SpU=Gamma}, the Arveson spectrum 
$Sp(\mathbf{U})$ is $ \mathbb R$. By the general result  of Evans \cite{Ev},
$Sp(\mathbf{U})=\sigma(B)$; hence (i) holds.
\end{proof}

{\bf Remarks.} 1. Since $\Sigma(E)=\sigma(B)$, \cite{Ev} yields $\Sigma(E)=Sp(\mathbf{U})$.

2. Foia\c s \cite[Corollaire 2]{Fo} proved that if $(U(t))_{t\in\mathbb R}$ is the 
unitary group induced on $L_2(\Omega,m)$ by an ergodic aperiodic measure preserving 
action $\{\theta(t) : t \in\mathbb R\}$, then the support of its spectral measure is 
$\mathbb R$; hence $\sigma(B)=\mathbb R$ (where $U(t)=\e^{itB}$). 
\bigskip

{\bf Acknowledgements.} The authors are grateful to Yuri Tomilov for useful comments
on Ritt operators, for pointing out \cite[Proposition 8.3.2]{Pe}, and for providing
 reference \cite{Ar2}. The second author thanks Eli Shamovich for discussions
 which led to Corollary \ref{shamo}.

\end{document}